\documentclass[12pt]{article}

\usepackage[margin=1.30in]{geometry}
\usepackage{amsmath,amssymb,amsthm,mathtools}
\usepackage{xcolor}
\usepackage{microtype}
\usepackage{authblk}
\usepackage[symbol]{footmisc}

\usepackage{newtxtext,newtxmath}
\usepackage[colorlinks=true,linkcolor=black,citecolor=black,urlcolor=black]{hyperref}
\usepackage[capitalize,nameinlink]{cleveref}
\usepackage{authblk}
\usepackage{bm}

\newcommand{\calP}{\mathcal P}

\newcommand{\ov}[1]{\overline{#1}}

\theoremstyle{plain}
\newtheorem{theorem}{Theorem}[section]
\newtheorem{lemma}{Lemma}[section]
\newtheorem{claim}{Claim}[section]

\newtheorem{corollary}{Corollary}[section]
\newtheorem{conjecture}{Conjecture}[section]
\newtheorem{question}{Question}[section]

\theoremstyle{definition}

\theoremstyle{definition}
\newtheorem{remark}{Remark}[section]

\title{$k$-arc-strong orientations of semicomplete digraphs}
\date{}
\author[1]{Tong Zhou}
\author[1,2]{J{\o}rgen Bang-Jensen}
\author[3]{Jia Zhou}
\author[1 *]{Jin Yan}
\affil[1]{School of Mathematics, Shandong University, Jinan 250100, China}
\affil[2]{Department of Mathematics and Computer Science, University of 

Southern Denmark, Odense DK-5230, Denmark}
\affil[3]{School of Mathematics and Statistics, Ningxia University, Yinchuan 750021, China}

\begin{document}
\maketitle
\footnotetext[1]{Corresponding author. E-mail address: yanj@sdu.edu.cn}
\begin{abstract}
Results by Jackson and Frank  imply  that every $2k$-arc-strong digraph $D$ contains a spanning $k$-arc-strong oriented subdigraph. This is best possible, even for very dense digraphs.
A digraph is {\bf semicomplete} if at least one of the arcs $xy,yx$ is present for every pair of  distinct vertices $x,y$.  A {\bf tournament} has exactly one of
$xy,yx$ for every such pair. Clearly every semicomplete digraph $D$ contains a spanning tournament $T$ which is obtained by deleting one arc from every 2-cycle of $D$.  
We prove that every $(2k-1)$-arc-strong semicomplete digraph on at least
$2k+1$ vertices contains a spanning $k$-arc-strong tournament.  Both bounds
$2k-1$ and $2k+1$ are best possible.  The proof uses  Frank's general orientation theorem for  graphs as well as counting arguments based on the semicomplete structure.
\end{abstract}


\section{Introduction}

Every graph $G=(V,E)$ has a corresponding digraph 
$\stackrel{\leftrightarrow}{G}$, called the {\bf complete biorientation of $G$}, which is obtained by replacing every edge $xy$ of $E$ by the directed 
2-cycle $xyx$. A digraph $D$ is {\bf symmetric} if $D= \stackrel{\leftrightarrow}{G}$ for some graph $G$.  An {\bf oriented graph} is a digraph with no directed
$2$-cycle. By an {\bf orientation} of a digraph $D$ we mean a spanning
oriented subdigraph obtained by retaining exactly one arc from every directed
$2$-cycle of $D$. Hence an orientation of a graph $G$ is the same as orienting its corresponding symmetric digraph $\stackrel{\leftrightarrow}{G}$.

A digraph $D=(V,A)$ is  \textbf{$\bm{k}$-arc-strong} if it remains strongly connected after the removal of any subset of $k-1$ arcs. A classical result by Robbins \cite{robbinsAMM46} says that a graph has a strong orientation if and only if it is 2-edge-connected. Nash-Williams \cite{NashWilliams1960} generalized this result as follows.

\begin{theorem}\cite{NashWilliams1960}
\label{thm:nwthm}
A graph $G$ has a $k$-arc-strong orientation $\stackrel{\rightarrow}{G}$ if and only if $G$ is $2k$-edge-connected.
\end{theorem}

Theorem \ref{thm:nwthm}, which  implies that a $2k$-arc-strong 
 symmetric digraph has a $k$-arc-strong orientation, was generalized to orientations 
of  arbitrary digraphs by  Jackson \cite{Jackson1988} who proved
the following result which can also be deduced  from a more general orientation theorem of Frank  \cite{Frank1996}.

\begin{theorem}\label{thm:general}\cite{Jackson1988}
Every $2k$-arc-strong digraph has a spanning $k$-arc-strong orientation.
\end{theorem}

Inspired by this result, Jackson and Thomassen made the following conjecture for the case of vertex-connectivity. A digraph $D=(V,A)$ is \textbf{$\bm{k}$-strong} if it has at least $k+1$ vertices and remains strong after the deletion of any subset of $k-1$ vertices. 

\begin{conjecture}\cite{Thomassen1989}
    Every $2k$-strong digraph has a spanning $k$-strong oriented subdigraph.
\end{conjecture}

For $k=1$ this follows from a result of Boesch and Tindell on strong orientations of mixed graphs \cite{boeschAMM87} but the conjecture is open even for $k=2$, except for symmetric digraphs where Thomassen proved the conjecture \cite{thomassenJCTB110}. For $k \ge 3$, Garamvölgyi et al. \cite{garamvolgyi2025highly} show that such a bound exists: every $320k^2$-connected undirected graph admits a $k$-strong orientation. This result applies to symmetric digraphs but does not extend to arbitrary digraphs.

While $2k$-edge-connectivity is clearly necessary in Theorem \ref{thm:nwthm}, implying that $2k$-arc-connectivity is best possible already for symmetric digraphs, 
the only lower bound on the arc-connectivity that holds for all digraphs having a $k$-arc-strong orientation is $k$, since the digraph may already be oriented. Thus, it is interesting to consider when $2k$-arc-connectivity is really necessary. 
\begin{remark}\label{trivial}
There are infinitely many digraphs (other than the symmetric ones) for which we cannot lower that arc-connectivity requirement. For example we can take two vertex disjoint $2k$-arc-strong digraphs $D_1,D_2$ and construct the digraph $D$ by adding $(2k-1)$ 2-cycles arbitrarily between the vertices of $D_1$ and $D_2$. Such a digraph is $(2k-1)$-arc-strong but contains no spanning oriented $k$-arc-strong digraph since the cut $(V(D_1),V(D_2))$ has less than $k$ arcs in one direction for every oriented subdigraph $H$ of $D$.
\end{remark}

Theorem \ref{thm:main} establishes that for semicomplete digraphs, the general $2k$-arc-connectivity sufficient condition can be improved to \(2k-1\), under the assumption \(n\ge 2k+1\). Both bounds are tight. 

\begin{theorem}\label{thm:main}
Let $k\ge1$, and let $D$ be a $(2k-1)$-arc-strong semicomplete digraph
on at least $2k+1$ vertices.  Then $D$ contains a spanning
$k$-arc-strong tournament.
\end{theorem}

\begin{remark}\label{rem:order-sharp}
The order $2k+1$ is best possible.  Indeed, the complete biorientation $\stackrel{\leftrightarrow}{K}_{2k}$ is $(2k-1)$-arc-strong, whereas every tournament
on $2k$ vertices has a vertex of outdegree at most $k-1$ and hence cannot be
$k$-arc-strong.
\end{remark}

\begin{remark}\label{rem:connectivity-sharp}
The arc-connectivity $2k-1$ is also best possible.  We recall the construction
used for the vertex-connectivity version in \cite{bangDM310}.  Let
$k\ge2$.  Let $U$ and $W$ be disjoint copies of the complete biorientation $\stackrel{\leftrightarrow}{K}_{2k-2}$, with
$U=\{u_1,\ldots,u_{2k-2}\},
W=\{w_1,\ldots,w_{2k-2}\}.$
Add the arcs $u_iw_i$ for $1\le i\le 2k-2$ and the arcs $w_i u_j$ for
$i\ne j$.  Finally, take an arbitrary nonempty tournament $C$, add all arcs
from $W$ to $C$, and add all arcs from $C$ to $U$.  Let $H$ be the resulting
semicomplete digraph.

It is easy to check that $H$ is $(2k-2)$-strong and hence also  $(2k-2)$-arc-strong.  Let $T$ be any spanning
tournament of $H$.  The tournament $T[U]$ has order $2k-2$, and hence some
$u_i\in U$ satisfies
$d_{T[U]}^+(u_i)\le k-2.$ The only possible out-neighbour of $u_i$ outside $U$ is $w_i$.  Consequently $d_T^+(u_i)\le k-1.$ Thus $T$ is not $k$-arc-strong.  Taking different nonempty tournaments $C$
gives an infinite family of $(2k-2)$-arc-strong semicomplete digraphs of
order at least $2k+1$ containing no spanning $k$-arc-strong tournament.
\end{remark}

The following result shows that the ($2k-1$)-arc-connectivity bound in Theorem~\ref{thm:main} fails even for digraphs that are almost semicomplete.
\begin{theorem}
\label{prop:deletepath}
 For every \(k\ge 2\), there exists an infinite family $\mathcal{D}$ of \((2k-1)\)-arc-strong digraphs, each of  which is obtained from a semicomplete digraph by deleting all arcs corresponding to a path of length at least $2k$ in the underlying complete graph, and with the property that  no digraph $D\in\mathcal{D}$ has a spanning \(k\)-arc-strong
oriented subdigraph. Moreover, \(V(D)\) can be partitioned into two sets
each of which induces a semicomplete digraph.
\end{theorem}

Theorem \ref{thm:main} also provides support to the following conjecture, by showing that the claim holds when we replace vertex-connectivity by arc-connectivity. A digraph is $k$-strong if it has at least $k+1$ vertices and remains strongly connected after the deletion of any subset of size $k-1$ of its vertices.

\begin{conjecture}\cite{bangDM310}
    Every $(2k-1)$-strong semicomplete digraph on at least $2k+1$ vertices contains a spanning $k$-strong tournament.
\end{conjecture}

To prove Theorem \ref{thm:main}, we apply Frank's general orientation theorem \cite{Frank1996} to a mixed graph derived from $D$, and crucially use the semicomplete structure to strengthen the sufficient arc-connectivity from $2k$ to $2k-1$. 
We construct this mixed graph by keeping all one-way arcs and replacing each directed 2-cycle with an undirected edge, then verify the required inequalities through a novel combination of a counting identity and discrete convexity arguments.

\section{Terminology and preliminaries}
Notation generally follows \cite{BangJensenGutin2009}. A set function $f:2^V\to\mathbb{R}$ is \textbf{submodular} if 
$
f(X)+f(Y)\ge f(X\cap Y)+f(X\cup Y)
$
holds for all subsets $X,Y\subseteq V$.
Let $G=(V,E)$ be an undirected graph. For any disjoint vertex sets $P,Q\subseteq V$, let $e_G(P,Q)$ denote the number of edges with one endpoint in $P$ and the other in $Q$. Based on this notation, for arbitrary $X,Y\subseteq V$ we define
$
d_G(X,Y)=e_G(X\setminus Y,Y\setminus X).
$
A function $h:2^V\to\mathbb{Z}\cup\{-\infty\}$ is \textbf{crossing $G$-supermodular} if the inequality
$
h(X)+h(Y)\le h(X\cap Y)+h(X\cup Y)+d_G(X,Y)
$
holds for every crossing pair $X,Y$.

For a digraph $D$ and $X\subseteq V(D)$, let
$
d_D^-(X)=|\{xy\in A(D):x\notin X,\ y\in X\}|
$
and
$d_D^+(X)=|\{xy\in A(D):x\in X,\ y\notin X\}|.$
For a vertex $v$, write $d_D^+(v)=d_D^+(\{v\})$ and
$d_D^-(v)=d_D^-(\{v\})$.  The arc-connectivity of a strong digraph $D$ is
$
\lambda(D)=
\min_{\varnothing\ne X\subsetneq V(D)}d_D^-(X).
$
Equivalently, by the directed arc version of Menger's theorem, $D$ is
$k$-arc-strong if and only if
$
d_D^-(X)\ge k
 (\varnothing\ne X\subsetneq V(D)).
$

A {\bf subpartition} of a set $V$ is a family of pairwise disjoint nonempty
subsets of $V$; its members need not cover $V$.  Two subsets $X,Y\subseteq V$
are {\bf crossing} if all four sets
$X\cap Y,
X\setminus Y,
Y\setminus X,
V\setminus(X\cup Y)$
are nonempty. For a set \(X\subseteq V\), we write $\ov{X}=V\setminus X.$
A family \(X_1,\ldots,X_r\) is {\bf pairwise co-disjoint} if
\(X_p\cup X_q=V\) whenever \(p\ne q\).  Equivalently, the complements
$\ov{X_1},\ldots,\ov{X_r}$
are pairwise disjoint.  In particular, if \(W=\bigcap_{j=1}^r X_j\), then
$
V\setminus W= V\setminus (\bigcap_{j=1}^r X_j)=\bigcup_{j=1}^r \ov{X_j}.$
Thus the nonempty members among \(\ov{X_1},\ldots,\ov{X_r}\) form a
partition of \(V\setminus W\).

Our proof of Theorem \ref{thm:main} is built upon Frank's general orientation theorem, which reduces the existence of the desired $k$-arc-strong tournament to a family of cut inequalities. To establish these inequalities via extremal counting arguments, we invoke a discrete convexity principle for integer partitions. The two lemmas are stated as follows.

\begin{lemma}[Discrete convexity, cf.~\cite{MarshallOlkinArnold2011}]
\label{lem:convex-majorization}
Let $f:\mathbb Z_{>0}\to\mathbb R$ be a function whose forward differences
$
\Delta f(s):=f(s+1)-f(s)$ for $s\ge1$
are nondecreasing.  Let $x_1,\ldots,x_m$ be positive integers with
$
x_1+\cdots+x_m=M.
$
Then
$
\sum_{\ell=1}^{m} f(x_\ell)
\le
f(M-m+1)+(m-1)f(1).
$\end{lemma}
\begin{corollary}\label{cor}
let \(a,b\) be integers with \(a\le b\), and
\(g(m)=f(M-m+1)+(m-1)f(1)\)
for $a\le m\le b.
$
Then
$
\max_{a\le m\le b} g(m)
=
\max\{g(a),g(b)\}.$
\end{corollary}


\begin{theorem}[Frank \cite{frank1993,Frank1996}]\label{lem:frank}
Let $G=(V,E)$ be an undirected graph, and $h:2^V\to\mathbb Z\cup\{-\infty\}$ a crossing $G$-supermodular set function with
$h(\varnothing)=h(V)=0$. Then $G$ has an orientation $\vec G$ satisfying
$d_{\vec G}^-(X)\ge h(X)$ for every $X\subseteq V$
if and only if
\begin{equation}\label{eq:frank}
s_t\ge
\sum_{i=1}^{t}
\left(
\sum_{j=1}^{r_i}h(V_{ij})-e_i
\right)
\end{equation}
for every subpartition $\{V_1,\ldots,V_t\}$ of $V$ where each $V_i$ is the intersection of a family of pairwise co-disjoint sets $V_{i1},\ldots,V_{ir_i}$, $s_t$ denotes the number of edges entering a $V_i$, and for a given $i$, $e_i$ denotes the number of edges connecting different sets $V_{ij}$.
\end{theorem}
More precisely, an edge
connecting $V_{ip}$ and $V_{iq}$ is counted between the disjoint sets
$V_{ip}\setminus V_{iq}$ and $V_{iq}\setminus V_{ip}$; thus
$
e_i=
\sum_{1\le p<q\le r_i}
e_G\bigl(V_{ip}\setminus V_{iq},\,
          V_{iq}\setminus V_{ip}\bigr).
$

\section{Proof of \cref{thm:main} }

Let $D=(V, A)$ satisfy the hypothesis of \cref{thm:main} and set
$n=|V|$, that is, $D$ is a \((2k-1)\)-arc-strong semicomplete digraph with $n\geq{2k+1}$. We apply Theorem~\ref{lem:frank} to a mixed graph associated with $D$. We first construct this mixed graph and translate the $k$-arc-strong requirement into an orientation constraint via a set function $h(X)$, thus reducing the proof  to verifying the inequality (\ref{eq:frank}) required by Theorem~\ref{lem:frank}.

Let
$
A_0=\{xy\in A:yx\notin A\}
$
be the set of arcs of $D$ which do not belong to a directed $2$-cycle, and write
$
D_0=(V,A_0). $Let $G=(V,E)$ be the undirected graph in which $xy\in E$ precisely when both
$xy$ and $yx$ belong to $A$.  Thus
$
M=(V,A_0\cup E)
$
is the mixed graph associated with $D$, whose directed part is $D_0$ and
whose undirected part is $G$.
An orientation of $G$ is an orientation of the undirected part of $M$: the
arcs of $A_0$ remain fixed, and one direction is chosen from each directed
$2$-cycle of $D$.  Hence every orientation $\vec G$ of $G$ gives a spanning
tournament
$
T=(V,A_0\cup A(\vec G))
$
contained in $D$.

Define $h:2^V\to\mathbb Z$ by
\[
h(X)=
\begin{cases}
k-d_{D_0}^-(X),&\varnothing\ne X\subsetneq V,\\
0,&X=\varnothing\text{ or }X=V.
\end{cases}
\]
If $d_{\vec G}^-(X)\ge h(X)$ for every $X\subseteq V$, then
$
d_T^-(X)=d_{D_0}^-(X)+d_{\vec G}^-(X)\ge k
$
for every nonempty proper $X\subset V$.  Hence $T$ is $k$-arc-strong.

It remains to verify the premise of Theorem~\ref{lem:frank}, that is, that the function $h(X)$ is crossing $G$-supermodular.  The in-cut
function $X\mapsto d_{D_0}^-(X)$ is submodular (See e.g. Section 5.1 in \cite{BangJensenGutin2009}). Thus, for crossing sets
$X,Y$,
$
d_{D_0}^-(X)+d_{D_0}^-(Y)
\ge
d_{D_0}^-(X\cap Y)+d_{D_0}^-(X\cup Y).
$
Since all four sets in the definition of crossing are nonempty, the sets
$X,Y,X\cap Y,X\cup Y$ are nonempty and proper. Substituting the definition
of $h$ gives
$
h(X)+h(Y)\le h(X\cap Y)+h(X\cup Y).
$
Thus $h$ is crossing $G$-supermodular.

We now prove the inequalities in Theorem~\ref{lem:frank}.  We choose
$
V_1,\ldots,V_t\
$
and pairwise co-disjoint families
$
V_{i1},\ldots,V_{ir_i}$ for $1\le i\le t$. If \(V_i=V\) for some \(i\), then, since \(\calP\) is a subpartition, we have
\(t=1\) and \(V_1=V\). Hence \(V_{1j}=V\) for all \(j\), and
\eqref{eq:frank} holds with equality. Thus we may assume that $V_i\subsetneq V$ for $1\le i\le t$. If \(V_{ij}=V\) for some \(i,j\), then deleting this set does not change
\(\bigcap_j V_{ij}\), while its contribution to both \(h(V_{ij})\) and \(e_i\)
is zero. Hence we may also assume that
$
V_{ij}\subsetneq V$ for $1\le i\le t,\ 1\le j\le r_i$. Thus,
with the notation
$\ov{V_{ij}}=V\setminus V_{ij},$
each set \(\ov{V_{ij}}\) is nonempty. In particular, for
$c_i=|V_i|
$ we have
$1\le c_i\le n-1.$
Put
$
R=\sum_{i=1}^{t}r_i.
$
Substituting the definition of $h$ into \eqref{eq:frank}, and using
$d_{D_0}^-(V_{ij})=d_{D_0}^+(\ov{V_{ij}}),$
it is enough to prove
\begin{equation}\label{eq:certificate}
s_t+
\sum_{i=1}^{t}e_i+
\sum_{i=1}^{t}\sum_{j=1}^{r_i}
d_{D_0}^+(\ov{V_{ij}})
\ge kR.
\end{equation}

We first rewrite the left-hand side of \eqref{eq:certificate}. Set
\[
U=\bigcup_{i=1}^{t}V_i,\ q=d_D^-(U),\ K=\sum_{1\le i<\ell\le t}c_i c_\ell,\ \text{and}\ P_i=\sum_{1\le p<q\le r_i}|\ov{V_{ip}}||\ov{V_{iq}}|.
\]

Thus, \(K\) is the number of unordered vertex-pairs with endpoints in two
distinct members of the subpartition \(V_1,\ldots,V_t\). Let \(P_i\) denote the number of unordered vertex-pairs with one endpoint in
\(\ov{V_{ip}}\) and the other in \(\ov{V_{iq}}\), in all distinct
\(p,q\in[r_i]\).

We will verify the inequality (\ref{eq:frank}) required in Theorem~\ref{lem:frank} in the following steps. First, by Claim~\ref{clm:count}, the left-hand side of \eqref{eq:certificate} reduces to \(\sum_{i=1}^t P_i + K + q\), so it suffices to verify inequality (\ref{eq:numerical}), and then for each subpartition \(V_i\), Claim~\ref{clm:local-deficit} invokes Lemma~\ref{lem:convex-majorization} to derive the lower bound \(P_i \ge kr_i - d(c_i)\), where \(d(c_i)\) denotes the maximum local deficit over all co-disjoint partitions of \(V\setminus V_i\). Finally, Claim~\ref{clm:deficit} establishes \(\sum_{i=1}^{t} d(c_i) \le K+q\). Combining these two estimates immediately yields \eqref{eq:numerical}.

\begin{claim}\label{clm:count}
The following holds.
\begin{equation}\label{eq:count}
s_t+
\sum_{i=1}^{t}e_i+
\sum_{i=1}^{t}\sum_{j=1}^{r_i}
d_{D_0}^+(\ov{V_{ij}})
=
\sum_{i=1}^{t}P_i+K+q.
\end{equation}
\end{claim}

\begin{proof}
For a given $i$, recall $V_{ip}\cup V_{iq}=V$ for $p\ne q$, we have
$
V_{ip}\setminus V_{iq}=\ov{V_{iq}},
V_{iq}\setminus V_{ip}=\ov{V_{ip}}.$
Thus
$
e_i=
\sum_{1\le p<q\le r_i}
e_G\bigl(\ov{V_{ip}},\ov{V_{iq}}\bigr).
$ First, consider 
$
e_i+\sum_{j=1}^{r_i}d_{D_0}^+(\ov{V_{ij}}).
$
Since $D$ is semicomplete, every vertex pair between $\ov{V_{ip}}$ and $\ov{V_{iq}}$ for distinct $p,q$ exists and is counted exactly once in $e_i+\sum_{j=1}^{r_i}d_{D_0}^+(\ov{V_{ij}})$: a one-way arc is counted in its tail set, while a directed $2$-cycle is counted once in $e_i$. All such pairs yield
$
P_i=\sum_{1\le p<q\le r_i} |\ov{V_{ip}}|\,|\ov{V_{iq}}|.
$
The remaining terms in $e_i+\sum_{j=1}^{r_i}d_{D_0}^+(\ov{V_{ij}})$ are precisely the arcs of $D_0$ entering $V_i$.
Consequently,
$
\sum_{i=1}^{t}(e_i+
\sum_{j=1}^{r_i}d_{D_0}^+(\ov{V_{ij}}))
=
\sum_{i=1}^{t}(P_i+d_{D_0}^-(V_i)).
$

We proceed to evaluate the sum
$
s_t+\sum_{i=1}^{t}d_{D_0}^-(V_i)$. It enumerates all undirected edges and one-way arcs entering $V_1, V_2,\cdots, V_t.$
Since $V_1,\ldots, V_t$ are pairwise disjoint, the sum splits naturally into two parts: the edges and arcs between $V_i$, $V_l$ for $i \neq l$, and arcs entering
$
U=\bigcup_{i=1}^{t}V_i
$
from $V\setminus U$.
Every vertex pair with endpoints in $V_i$ and $V_\ell$ is counted exactly once in $
s_t+\sum_{i=1}^{t}d_{D_0}^-(V_i)
$: a one-way arc enters exactly one of the two sets, while a directed 2-cycle is counted once in $s_t$. These pairs are precisely enumerated by $K=\sum_{1\le i<\ell\le t}c_i c_\ell$.

The rest of the sum counts exactly the arcs of $D$ entering $U$ from $V\setminus U$, that is
$
q=d_D^-(U).
$
Combining these two parts gives
$
s_t+\sum_{i=1}^{t}d_{D_0}^-(V_i)=K+q.
$
Together with $
\sum_{i=1}^{t}(e_i+
\sum_{j=1}^{r_i}d_{D_0}^+(\ov{V_{ij}}))
=
\sum_{i=1}^{t}(P_i+d_{D_0}^-(V_i))$, this yields
$
s_t+
\sum_{i=1}^{t}e_i+
\sum_{i=1}^{t}\sum_{j=1}^{r_i}
d_{D_0}^+(\ov{V_{ij}})
=
\sum_{i=1}^{t}P_i+K+q,
$
as required.
\end{proof}
By Claim~\ref{clm:count}, the remaining task is to prove
\begin{equation} \label{eq:numerical}
    \sum_{i=1}^{t}P_i+K+q\ge kR.
\end{equation}

We prove this in two steps.  First, we obtain local lower bounds for $P_i$.  Then we show that the resulting deficits are bounded by
$K+q$.

\begin{claim}\label{clm:local-deficit}
For $1\le c\le n-1$, define
$
d(c)=
\max\left\{
k,(n-c)(2k+1-n+c)/2
\right\}.$
\footnote{The two values in the maximum correspond to the two extremal partitions of $V\setminus V_i$: when $r_i=1$, we have $V\setminus V_i=\ov{V_{i1}}$; when $r_i=n-c_i$, we have $|\ov{V_{ij}}|=1$ for all $j\in[r_i]$.}
Then $
P_i\ge kr_i-d(c_i)
$ for every $i$.

\end{claim}

\begin{proof}
For every $i$, put
$
m_i=n-c_i, b_j=|\ov{V_{ij}}|$ for $1\le j\le r_i.
$
The sets
$
\ov{V_{i1}},\ldots,\ov{V_{ir_i}}
$
are nonempty and a partition of $V\setminus V_i$. Hence
$
\sum_{j=1}^{r_i} b_j=m_i.
$
Moreover,
$
P_i=
\sum_{1\le p<q\le r_i}b_pb_q
=
\frac12\left(m_i^2-\sum_{j=1}^{r_i}b_j^2\right).
$
Applying Lemma~\ref{lem:convex-majorization} to $f(x)=x^2$, the sum
$\sum_{j=1}^{r_i}b_j^2$ is maximized when the sizes are
$
m_i-r_i+1,1,\ldots,1.
$
Therefore
\begin{equation}\label{eq:Pi-lower}
P_i
=
\frac12\left(m_i^2-\sum_{j=1}^{r_i}b_j^2\right)
\ge
\frac12\left(m_i^2-(m_i-r_i+1)^2-(r_i-1)\right)
=
\frac{(r_i-1)(2m_i-r_i)}2 .
\end{equation}
Consequently,
$
kr_i-P_i
\le
kr_i-\frac{(r_i-1)(2m_i-r_i)}2.
$
For fixed \(c\), put \(m=n-c\) and define
$
F_c(r)
=
kr-\frac{(r-1)(2m-r)}2$ for $1\le r\le m.$
Then
$
F_c(r)
=
\frac{r^2}{2}+\left(k-m-\frac12\right)r+m,
$
so \(F_c\) is a convex quadratic function of \(r\). Hence its maximum on
the integer interval \(1\le r\le m\) is attained at one of the endpoints.
The two endpoint values are
$
F_c(1)=k
$
and
$
F_c(m)
=
\frac{m^2}{2}+
km-m^2-\frac{m}{2}+m
=
\frac{m(2k+1-m)}2
=
\frac{(n-c)(2k+1-n+c)}2.
$
Therefore
\[
kr_i-P_i\le
\max\left\{
k,\frac{(n-c_i)(2k+1-n+c_i)}2
\right\}
=d(c_i),
\]
as required.

\end{proof}

The remaining point is to show that the total deficit in
Claim~\ref{clm:local-deficit} is bounded by $K+q$.

\begin{claim}\label{clm:deficit-bound}
The inequality
$
\sum_{i=1}^{t}d(c_i)\le K+q
$ \label{clm:deficit}
holds.
\end{claim}

\begin{proof}

Define $\phi(c)=d(c)+c^2/2$, and set $C=|U|=\sum_{i=1}^t c_i$. Observe that
$
K=\sum_{1\le i<\ell\le t}c_ic_\ell=(C^2-\sum_{i=1}^t c_i^2)/2 .
$
we have
\begin{equation}\label{eq:deficit-rewrite}
\sum_{i=1}^{t}d(c_i)-K
=
\sum_{i=1}^{t}\left(d(c_i)+\frac{c_i^2}{2}\right)
-\frac{C^2}{2}=\sum_{i=1}^{t}\phi(ci)-\frac{C^2}{2}.
\end{equation}
Recall $d(c)=kr_i-P_i\le
\max\left\{
k,\frac{(n-c_i)(2k+1-n+c_i)}2
\right\}$, we have
\[
\begin{aligned}
\phi(c)
&=d(c)+\frac{c^2}{2} =
\max\left\{
k+\frac{c^2}{2},
\frac{(n-c)(2k+1-n+c)+c^2}{2}\right\} \\
&=\max\left\{
	k+\frac{c^2}{2},
	\frac{n(2k+1-n)+(2n-2k-1)c}{2}
	\right\}.
\end{aligned}
\]
Thus $\phi$ is convex on the positive integers.
We prove \eqref{clm:deficit} separately for $U\subsetneq V$ and $U=V$.

\medskip
\noindent\textbf{Case 1: $U\subsetneq V$}, and then $C<n$.  

Since $V_1,\ldots,V_t$ are nonempty and pairwise disjoint,
and $
1\le t\le C.
$
Applying Lemma~\ref{lem:convex-majorization} to the convex function $\phi$
gives
$
\sum_{i=1}^{t}\phi(c_i)
\le
\max\{\phi(C),C\phi(1)\}.
$
Using \eqref{eq:deficit-rewrite}, we obtain
\begin{equation}\label{eq:two-extremes}
	\sum_{i=1}^{t}d(c_i)-K
	\le
	\max\left\{
	\phi(C)-\frac{C^2}{2},
	C\phi(1)-\frac{C^2}{2}
	\right\}.
\end{equation}
We now evaluate the two terms on the right-hand side.
For the second term, Recall $d(c)
\le
\max\left\{
k,(n-c_i)(2k+1-n+c_i)/2
\right\}$ and $\phi(c)=d(c)+c^2/2$.
\[\phi(1)=d(1)+1/2
=
\max\left\{
k,\frac{(n-1)(2k+2-n)}2
\right\}+1/2=k+1/2.\]
since $n\ge 2k+1$, $(n-1)(2k+2-n)/2\le k$ holds,
It follows that
\[
C\phi(1)-\frac{C^2}{2}
=
C\left(k+\frac12\right)-\frac{C^2}{2}
=
\frac{C(2k+1-C)}2.
\]
Consequently,
\begin{equation}\label{eq:extremes}
	\sum_{i=1}^{t}d(c_i)-K
	\le
	\max\left\{
	d(C),\frac{C(2k+1-C)}2
	\right\}.
\end{equation}

It remains to prove that $q$ is greater than the two terms in \eqref{eq:extremes}.
We first prove
\begin{equation}\label{eq:q-C}
q\ge \frac{C(2k+1-C)}2.
\end{equation}
Since $U$ is a nonempty proper subset of $V$ and
$D$ is $(2k-1)$-arc-strong, then $q=d_D^-(U)\ge2k-1$ holds. In particular, every vertex has indegree and outdegree at least $2k-1$.
Summing indegrees over $U$, we obtain
$
C(2k-1)
\le
\sum_{u\in U}d_D^-(u)
\le
|A(D[U])|+q.
$
As $D[U]$ has at most $C(C-1)$ arcs, this gives
$
q\ge C(2k-C).
$
If $C\le 2k-1$, then
\[
C(2k-C)-\frac{C(2k+1-C)}2
=
\frac{C(2k-1-C)}2\ge0,
\]
and \eqref{eq:q-C} follows.  If $C=2k$, then the right-hand side of
\eqref{eq:q-C} is $k$, while $q\ge2k-1\ge k$.  If $C\ge2k+1$, then the right-hand side of
\eqref{eq:q-C} is nonpositive.  Hence \eqref{eq:q-C} holds in all cases.

We next prove
$q\ge d(C)$.\label{eq:q-dC}
Put
$
C'=n-C=|V\setminus U|.
$
Then
$
q=d_D^-(U)=d_D^+(V\setminus U).
$
Summing outdegrees over $V\setminus U$, and arguing as above, gives
\[
C'(2k-1)
\le
\sum_{v\in V\setminus U}d_D^+(v)
\le
|A(D[V\setminus U])|+q
\le
C'(C'-1)+q.
\]
Thus
$
q\ge C'(2k-C').
$
Recall
$
d(C)=
\max\left\{
k,(C'(2k+1-C')/2
\right\}.$
The first term is bounded by $q$, since $q\ge2k-1\ge k$. For the second term, the same comparison as above, with
$C'$ in place of $C$. Combining \eqref{eq:extremes}, \eqref{eq:q-C}, and $q\ge2k-1$, we get
$
\sum_{i=1}^{t}d(c_i)-K\le q.
$ This proves the claim when $U\subsetneq V$.

\medskip
\noindent\textbf{Case 2: $U=V$}, and then $C=n, q=0.$

It suffices to prove
$
\sum_{i=1}^{t}d(c_i)-K\le0.
$ If $t=1$, then $U=V$ implies $V_1=V$.  Since
$V_1=\bigcap_{j=1}^{r_1}V_{1j}$, this forces $V_{1j}=V$ for all $j$.
Consequently $h(V_{1j})=h(V)=0$, $e_1=0$, and $s_t=0$, so
\eqref{eq:frank} holds with equality. Hence we may assume $t\ge2$.
Since $V_1,\ldots,V_t$ are nonempty and pairwise disjoint and their union is
$V$, we have $2\le t\le n$.

Applying Lemma~\ref{lem:convex-majorization} to $\phi$, with $C=n$, gives
\[
\sum_{i=1}^{t}\phi(c_i)
\le
\max\{\phi(n-1)+\phi(1),n\phi(1)\}.
\]
By \eqref{eq:deficit-rewrite} we have
\[\sum_{i=1}^{t}d(c_i)-K
\le
\max\left\{
\phi(n-1)+\phi(1)-\frac{n^2}{2},
n\phi(1)-\frac{n^2}{2}
\right\}.
\]
Now $d(1)=k$, as above, and $d(n-1)=
\max\left\{
k,\frac{1\cdot 2k}{2}
\right\}
=k.
$
Hence
\[
\phi(n-1)+\phi(1)-\frac{n^2}{2}
=
d(n-1)+d(1)-(n-1)
=
2k-n+1\le0,
\]
where the last inequality uses $n\ge2k+1$.  Similarly,
\[
n\phi(1)-\frac{n^2}{2}
=
n\left(k+\frac12\right)-\frac{n^2}{2}
=
\frac{n(2k+1-n)}2\le0,
\]
again by $n\ge2k+1$.  Therefore
$\sum_{i=1}^{t}d(c_i)-K\le0$.
\end{proof}

We now finish the proof of \eqref{eq:numerical}.  By
Claim~\ref{clm:local-deficit} and Claim~\ref{clm:deficit}, the following holds.
\[
\sum_{i=1}^{t}P_i+K+q
\ge
kR-\sum_{i=1}^{t}d(c_i)+K+q
\ge
kR.
\]
Consequently, this proves \eqref{eq:certificate}. Therefore the condidition (\ref{eq:frank}) in Theorem~\ref{lem:frank} is satisfied and we conclude that
$G$ has an orientation $\vec G$ such that
$
d_{\vec G}^-(X)\ge k-d_{D_0}^-(X)
$
for every nonempty proper subset $X\subset V$.  The resulting spanning
tournament
$
T=(V,A_0\cup A(\vec G))
$
satisfies
$
d_T^-(X)
=
d_{D_0}^-(X)+d_{\vec G}^-(X)
\ge k
$
for every nonempty proper $X\subset V$.  Thus $T$ is $k$-arc-strong.
This completes the proof of  \cref{thm:main}.
\hfill $\Box$

\section{Proof of Theorem \ref{prop:deletepath}}

The construction which we used in Remark \ref{trivial} to show that $(2k-1)$-arc-strong connectivity is not in general sufficient to guarantee a $k$-arc-strong orientation is 'trivial' in the sense that it contains a vertex cut $(X,V\setminus X)$ with just $2k-1$ connections (formed by 2-cycles) between the two sets. Clearly this implies that every orientation of those 2-cycles will result in either too few arcs in one direction between $X$ and $V\setminus X$. The construction below avoids this and still shows that $2k$-arc-connectivity is best possible in Theorem \ref{thm:general}. We will construct an infinite family of 
$(2k-1)$-arc-strong digraphs, each derived from a semicomplete digraph by removing an underlying path of length at least 
$2k$. Each digraph $D$ in the family admits a vertex partition into two induced semicomplete subdigraphs but $D$ has no spanning $k$-arc-strong oriented subdigraph.
\begin{proof}
Let $
U=\{u_0,u_1,\ldots,u_{2k}\}, W=\{w_1,\ldots,w_{2k}\}$.
We construct \(D\) as follows; see Figure~\ref{fig:construction} for an illustration. The subdigraph \(D[U]\) is obtained from the complete digraph on \(U\) by
deleting both arcs corresponding precisely to the edges of the underlying
path
$
u_0u_1\cdots u_{2k},
$ that is, we delete both arcs between each consecutive pair
$u_0u_1,\ u_1u_2,\ \ldots,\ u_{2k-1}u_{2k}$. The subdigraph $D[W]$ is the complete digraph on $W$.
Between \(U\) and \(W\), first put all arcs from \(U\) to \(W\). Then,
for each \(1\le i\le 2k-1\), replace the arc \(u_iw_i\) by the opposite
arc \(w_i u_i\). Thus the only arcs from \(W\) to \(U\) are 
$w_i u_i$ for $1\le i\le 2k-1$.
In particular, $d_D^-(U)=2k-1$.

\begin{figure}[htbp]
\centering
\includegraphics[width=0.8\textwidth]{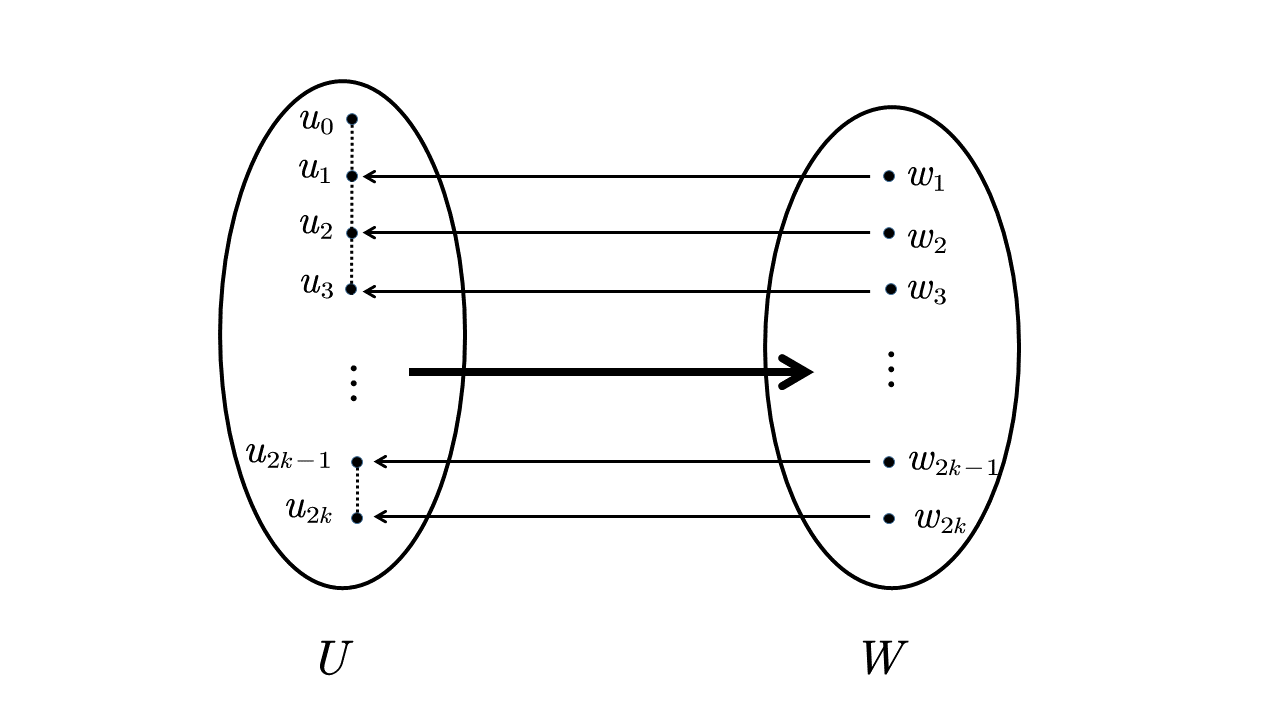}
\caption{The digraph $D$ in Theorem~\ref{prop:deletepath}. 
Dashed lines inside $U$ represent the underlying path edges whose bidirectional arcs are removed.
All arcs between $U$ and $W$ are directed from $U$ to $W$ (marked by thick arrows), except for the reverse arcs $w_iu_i$ ($1\le i\le 2k-1$) shown as thin arrows.
The induced subdigraph $D[W]$ is complete.} 
\label{fig:construction}
\end{figure}

It is easily to check  \(D\) is \((2k-1)\)-arc-strong.
We now show that \(D\) has no spanning \(k\)-arc-strong oriented
subdigraph. Suppose, to the contrary, that \(T\) is such a subdigraph.
Then
$
d_T^-(u_i)\ge k$
for $0\le i\le 2k$.
Hence $
\sum_{i=0}^{2k} d_T^-(u_i)\ge k(2k+1)$. On the other hand, since \(T\) is oriented, the number of arcs of \(T\)
entering vertices of \(U\) is at most the number of available underlying
pairs inside \(U\), plus the number of fixed arcs from \(W\) to \(U\).
Thus it is at most
\[
\binom{2k+1}{2}-2k+(2k-1)
= k(2k+1)-1,
\] which yields a contradiction. Hence \(D\) has no spanning \(k\)-arc-strong oriented
subdigraph.

Finally, partition the path vertices by parity: $U_{\rm even}=\{u_i:i\text{ is even}\}, U_{\rm odd}=\{u_i:i\text{ is odd}\}$. Set $S_1=W\cup U_{\rm even}, S_2=U_{\rm odd}$.
It is easy to check 
$D[S_1]$
and
$D[S_2]$ are semicomplete digraphs. Moreover, by increasing the size of \(W\) while keeping \(|W|\ge 2k\), the same construction gives an infinite family of such digraphs that satisfy the proposition.
\end{proof}
\section{Conclusion}
We proved that every $(2k-1)$-arc-strong semicomplete digraph on at least
$2k+1$ vertices contains a spanning $k$-arc-strong tournament. It is then natural to ask how far the semicomplete assumption can be relaxed. Proposition~\ref{prop:deletepath} shows that the conclusion fails even when these pairs are precisely the edges of a path. This suggests considering the more restricted case in which the nonadjacent pairs form a matching. In
particular, it remains natural to determine whether every $(2k-1)$-arc-strong
digraph that can be made semicomplete by adding arcs between the ends of a
matching, or even a perfect matching, contains a spanning $k$-arc-strong oriented subdigraph.
Let $\mathcal{D}_M$ denote the class of digraphs that can be made semicomplete by adding the arcs of a perfect matching.
\begin{question}
    Does every $(2k-1)$-arc-strong digraph in $\mathcal{D}_M$ contain a spanning $k$-arc-strong oriented subdigraph?
\end{question}

\end{document}